\documentclass[11pt]{article}
\usepackage{fullpage}
\usepackage{fancyhdr}
\usepackage{amsmath,amssymb,amsthm}
\usepackage{amsmath,amscd}
\usepackage{mathtools}
\usepackage{mathrsfs}
\usepackage{tikz-cd}
\usepackage{authblk}
\usepackage{enumitem}

\usepackage{stmaryrd}

\usepackage[letterpaper, margin=1in]{geometry}
\usepackage{baskervald}
\usepackage{euler}

\newcommand{\Z}{\mathbb{Z}}

\newcommand{\C}{\mathbb{C}}

\newcommand{\Q}{\mathbb{Q}}

\newcommand{\G}{\mathbb{G}}

\newcommand{\Hom}{\text{Hom}}
\newcommand{\Spec}{\text{Spec}}

\newcommand{\dep}[1]{\text{dep}(#1)}

\newcommand{\F}{\mathbb{F}}
\newcommand{\ppartial}[1]{\partial^{(p^n)}}
\newcommand{\powerseries}[1]{[\![#1]\!]}

\newcommand{\ipartial}[1]{\partial^{(#1)}}
\newcommand{\dlog}[1]{\textnormal{dLog}^{(#1)}}

\newcommand{\AHp}{\overline{E}_p}

\newcommand{\LS}[1]{(\!(#1)\!)}
\newcommand{\fracfield}[1]{\textnormal{Frac}(#1)}
\newcommand{\ID}{\textnormal{ID}}
\newcommand{\IF}{\textnormal{IF}}
\newcommand{\cat}{\textnormal{cat}}

\newcommand{\Gal}{\textnormal{Gal}}

\newtheorem{theorem}{Theorem}[section]
\newtheorem{lemma}[theorem]{Lemma}

\newtheorem{proposition}[theorem]{Proposition}

\newtheorem{corollary}[theorem]{Corollary}

\theoremstyle{definition}

\newtheorem{remark}[theorem]{Remark}
\newtheorem{definition}[theorem]{Definition}

\def\frak{\relaxnext@\ifmmode\let\next\frak@\else
 \def\next{\Err@{Use \string\frak\space only in math mode}}\fi\next}
\def\goth{\relaxnext@\ifmmode\let\next\frak@\else
 \def\next{\Err@{Use \string\goth\space only in math mode}}\fi\next}
\def\frak@#1{{\frak@@{#1}}}
\def\frak@@#1{\noaccents@\fam\euffam#1}
\font\tengoth=eufm10
\newfam\gothfam \def\goth{\fam\gothfam\tengoth} \textfont\gothfam=\tengoth

\title{Transcendence properties of the Artin-Hasse exponential modulo $p$}
\author{Joe Kramer-Miller}

\date{}

\begin{document}
	\maketitle
	\begin{abstract}
		Let $E_p(x)$ denote the Artin-Hasse exponential and let $\AHp(x)$ denote
		its reduction modulo $p$ in $\F_p\powerseries{x}$. 
		In this article we study transcendence properties of $\AHp(x)$ over $\F_p[x]$. 
		We give two proofs that $\AHp(x)$ is transcendental, affirmatively answering
		a question of Thakur. We also prove algebraic independence results:
		i) for $f_1,\dots,f_r \in x\F_p[x]$ satisfying certain linear independence properties, we
		show that the $\AHp(f_1), \dots, \AHp(f_r)$ are algebraically independent over $\F_p[x]$ and ii) we determine the algebraic relations between $\AHp(cx)$, where $c \in \F_p^\times$.
		Our proof studies the higher derivatives of $\AHp(x)$ and makes use of iterative differential Galois theory.
	\end{abstract}
	
	\section{Introduction}
	The Artin-Hasse exponential is defined by
	\begin{align*}
		E_p(x) := \exp\left(\sum_{n=0}^\infty \frac{x^{p^n}}{p^n}\right )
	\end{align*}
	and is known to have coefficients in $\Z_{(p)}$. 
	This series plays an enormous role in $p$-adic analysis, showing up
	in the construction of Witt vectors (see e.g. \cite{Dieudonne-ArtinHasse_exponential}) and the $p$-adic
	approach to exponential sums (see e.g. \cite{Adolphson-Sperber-exponential_sums} and \cite{Wan-NP_zeta_functions}). From the relation $\frac{E_p(x^p)}{E_p(x)^p}=\exp(-xp)$
	it is immediate that $E_p(x)$ is a transcendental function over $\C(x)$.
	However, we may also consider the reduction of $E_p(x)$ modulo $p$, which we denote by $\overline{E}_p(x)$. Then $\overline{E}_p(x)$ is a power series in $\F_p\powerseries{x}$
	that in some sense plays the role of the exponential function in characteristic $p$. 
	For example, one can use $\overline{E}_p(x)$ to translate Witt vector addition into power
	series multiplication as explained in \cite{Dieudonne-ArtinHasse_exponential}. In this article
	we investigate the transcendence properties of $\overline{E}_p(x)$. We give two proofs that $\overline{E}_p(x)$
	is transcendental over $\F_p(x)$. This answers a question posed by Dinesh Thakur (see e.g. \cite{Thakur-automatic_methods_in_transcendence} or the open questions in chapter 12 of \cite{Allouche-Shallit-Automatic_sequences}). We also prove  algebraic independence results for
	the evaluations of $\overline{E}_p(x)$ at different values in $x\overline{\F}_p[x]$.

	\subsection{Main results and outline}
	Our first result is the following:
	\begin{theorem} \label{t: basic transcendence result}
		The series $\AHp(x)$ is transcendental over $\F_p[x]$.
	\end{theorem}
	We give two proofs Theorem \ref{t: basic transcendence result}. The first
	proof is given in \S \ref{s: an automatic approach to the transcendence}. It uses
	Christol's automatic criterion for algebraicity of power series together with
	the functional equation: for $c \in \F_p^\times$ we have $\AHp(cx)=\AHp(x)^{[c]}$ for $c \in \F_p^\times$ 
	where $[c]$ denotes the Teichmuller lift of $c$. This proof is relatively elementary and short, but has the disadvantage of requiring $p$ to be at least $5$. The second proof 
	uses iterative differential Galois theory and iterative Frobenius equations. These
	are the same methods we use for the stronger algebraic independence results, which we
	describe in detail below.

	To state our first algebraic independence result we need the following definition.
	\begin{definition}
		Let $f(x)=\sum c_i x^i \in \overline{\F}_p[x]$. We define the \emph{primitive}
		part of $f$ to be
		\begin{align*}
		f^{*}(x) := \sum_{\substack{i \geq 0 \\ p\nmid i}} c^{*}_i(f) x^i,\\
		c^{*}_i(f):= \sum_{k\geq 0} c_{ip^k}^{1/p^k}.
		\end{align*}
	\end{definition}
	\noindent Note that $f^{*}(x)=0$ if and only if $f(x)=g(x)^p - g(x)$ for some $g(x) \in \F_p[x]$.
	\begin{theorem}
		\label{t: algebraic independence result}
		Let $f_1,\dots, f_r \in x\overline{\F}_p[x]$ be nonzero. Assume that one of the two conditions holds:
		\begin{enumerate}
			\item The elements $f_1^{*}, \dots, f_r^{*}$ are linearly independent over
			$\F_p$. 
			\item The $f_i^{*}$'s are not all zero and no non-trivial power product of the $f_i's$ are contained in $\F_p(x)^p$. In other words, if $f_1^{n_1}\dots f_r^{n_r}$ is a $p$-th power then $n_1,\dots,n_r \in p\Z$. 
		\end{enumerate}
		Then $\overline{E}_p(f_1), \dots, \overline{E}_p(f_r)$ are algebraically independent.
		\end{theorem}
		Our second algebraic independence result is the following. 
		\begin{theorem}
			\label{t: automata result}
			We have
			\begin{align*}
			\dim_{\F_p(x)} \F_p(x)(\overline{E}_p(cx))_{c \in \F_p^\times} &= \phi(p-1), 
		\end{align*}
		where $\phi$ is the Euler totient function.
		\end{theorem}
		The proofs of Theorem \ref{t: algebraic independence result} and Theorem
		\ref{t: automata result} follow the same strategy. This strategy can be broken into
		three steps, which we outline below. For simplicity, we restrict our attention
		to the case where the $f_i$ in Theorem \ref{t: algebraic independence result}
		have coefficients in $\F_p$.
		
		\paragraph{Iterated Frobenius equations.}
		Let $\sigma$ be the lifted Frobenius endomorphism of $\Z_p\powerseries{x}$ that fixes $\Z_p$ and sends $x$ to $x^p$. 
		An iterated Frobenius (abbreviated $\IF$) equation with depth $k$ is the system of equations: 
		\[y_1^\sigma - y_1 = x^{b_1}, ~~ y_2^\sigma - y_2 = x^{b_2}y_1, ~~  \dots,~~y_k^\sigma - y_k=x^{b_k}y_{k-1}.\]
		where the $b_i \geq 1$. We say an $\IF$-equation is primitive if each $b_i$ is prime to $p$. 
		Let $M_k$ denote the $\Z_p(x)$-module space generated by all solutions to all $\IF$-equations of depth at most $k$ and let $M=\cup M_k$. In \S \ref{s: solutions to iterated frobenius equations} we study $M$. We prove the following:
		\begin{enumerate}
			\item The solutions to primitive $\IF$-equations form a basis of $M$ over $\Z_p(x)$.
			\item The $\Z_p(x)$-module $M$ is closed under multiplication and $M_{k_1}M_{k_2} \subset M_{k_1+k_2}$. 
		\end{enumerate}
		
		\paragraph{Higher derivatives}
		Let $\ipartial{k}$ denote the differential operator $\frac{1}{k!}\frac{d^k}{d^kx}$
		and let $\dlog{k}$ be the operator sending $f$ to $\frac{\ipartial{k} f}{f}$. We regard
		$\ipartial{k}$ and $\dlog{k}$ as operators on $\F_p\powerseries{x}$ and on $\Z_p\powerseries{x}$. Note that if $f$ is contained in a finite extension $L$
		of $\F_p(x)$, then so is $\dlog{k}(f)$ for all $k\geq 0$. This gives a differential
		transcendence criterion: if the $\dlog{k}(f)$ generate an infinite extension of $\F_p(x)$ then $f$ is transcendental over $\F_p(x)$. In \S \ref{s: higher derivatives and the artin-hasse exponential} we study the iterated differential equations satisfied by $\overline{E}_p(x)$
		and power products of the $\overline{E}_p(f_r)$. We prove that $\dlog{k}(\AHp(x))$ is contained in $M_k$,
		but not in $M_{k-1}$. For example, when $k=1$ a quick calculation shows
		that $\frac{d}{dx}\overline{E}_p(x)=\frac{y}{x}\overline{E}_p(x)$, where $y^\sigma - y = x$.
		More precisely, we determine $\dlog{k}(\AHp(x))$ modulo $M_{k-1}$. This is
		enough to show $\AHp(x)$ is transcendental. Next,
		using standard properties of higher derivatives and results from \S \ref{s: solutions to iterated frobenius equations}, we are able to determine $\dlog{k}\left(\overline{E}_p(f_1)^{a_1}\dots \overline{E}_p(f_r)^{a_r}\right)$
		modulo $M_{k-1}$, where $a_i \in \Z$ and $f_i \in \F_p[x]$. 
		
		\paragraph{Iterative differential Galois groups}
		Let $K$ be the algebraic closure of $\F_p(x)$ in $\overline{\F}_p\LS{x}$. Note
		that $M$ is contained in $K$. An $\ID$-module over $K$ is
		a vector space over $K$ with compatible actions of $\ipartial{k}$. For any $f \in x\F_q[x]$ we let $N_f$ denote the $\ID$-module associated to the differential equation satisfied by
		$\AHp(f)$ (see \S \ref{ss: ID Galois theory and transcendence} for a precise
		description of this $\ID$-module). Following Matzat and van der Put (see \cite{Matzat-vanderPut-Iterative_diff_eqs}) we may define an $\ID$-Galois group
		$\Gal(N_f)$. Then $\Gal(N_f)$ is a closed subgroup of the multiplicative group $\G_m$ over $\overline{\F}_p$ and is equal to all of $\G_m$ precisely when $\AHp(f)$ is transcendental over $\F_p(x)$. More generally, 
		the algebraic independence of $\AHp(f_1), \dots, \AHp(f_r)$ is equivalent
		to the $\ID$-Galois group $\Gal(N_{f_1} \oplus \dots \oplus N_{f_r})$ being isomorphic
		to $\G_m^{r}$. Using the Tannakian interpretation of $\ID$-Galois groups, it is enough
		to prove $\overline{E}_p(f_1)^{a_1}\dots \overline{E}_p(f_r)^{a_r}$ is transcendental
		for any $a_i \in \Z$. We then address this transcendentce question under the conditions
		of Theorem \ref{t: automata result} and Theorme \ref{t: algebraic independence result}
		using the results from \S \ref{s: higher derivatives and the artin-hasse exponential}.

		\subsection{Further questions}
		It would be interesting to completely determine the algebraic relations between
		values $\AHp(f)$ for $f \in x \overline{\F}_p[x]$. More generally, it would be
		interesting to determine the algebraic relations between values of $\AHp(x)$ at
		any algebraic power series. At the moment it is even unclear to the author
		what the algebraic relations are between the $\AHp(\xi x)$ where $\xi$ ranges over  $ \overline{\F}_p$.  The reader may observe the influence of Ax's work on Schanuel's conjecture in \cite{Ax-schanuels_conjecture} on the final steps of the proof of Theorem \ref{t: algebraic independence result}.
		Thus, it is natural to ask if an analogue of Ax-Schanuel hold for the $\AHp(x)$.

		\subsection{Acknowledgments}
		The author wishes to thank Dinesh Thakur for his enthusiasm and his comments on some very preliminary versions.

	\section{An automatic approach to the transcendence of $\AHp(x)$}
	\label{s: an automatic approach to the transcendence}
	In this section we give a short proof of the transcendence of $\AHp(x)$ using
	the functional equation satisfied by $\AHp(x)$ and automatic methods.
	\subsection{A transcendental criterion}
	
	We recall Christol's well-known automatic criterion for algebraicity of series in $\F_p\powerseries{x}$. See \cite{Christol} or \cite{Christol-Kamae-France-Rauzy} for
	proofs. 
	\begin{theorem}
		\label{t: algebraic criterion via automata}
		Let $f=\sum f_n x^n \in \F_p\powerseries{x}$. The following are equivalent. 
		\begin{enumerate}
			\item The series $f$ is algebraic over $\F_p[x]$. 
			\item There is a $p$-automaton that gives output $f_n$ on input $n$.
			\item There are only finitely many subsequence of $(f_n)_{n \geq 0}$ of the form
			$(f_{p^k n + r})_{n \geq 0}$, where $k$ is positive and $r$ is in the range $0\leq r < p^k$. 
		\end{enumerate}
	\end{theorem}
	This allows us to establish the following transcendence criterion:
	\begin{lemma}
		\label{l: transcence of $(1+x)^c$} Let $\lambda \in \Z_p$. Let $f \in 1 + x\F_p\powerseries{x}$ be algebraic over $\F_p[x]$. Then $f^\lambda \in \F_p\powerseries{x}$ is algebraic over $\F_p[x]$ if and only if $\lambda$ is rational. 
	\end{lemma}
	\begin{proof}
		This is a special case of \cite{Allouche-MendesFrance-vanderPoorten}. We will give a proof in the interest of having a quick and self-contained automatic proof that $\AHp(x)$ is transcendental. We easily reduce to the case where $f = 1+x$. If $\lambda$ is rational it is clear that $(1+x)^\lambda$ is algebraic. For the other direction, assume $\lambda$ is irrational. Write $\lambda = \sum_{n=0}^\infty a_n p^n$. For $1 \leq c \leq p-1$ we define sets
		\[S_c := \left\{n ~|~ a_n = c \right\}  ~~\text{and}~~ S_c^{>k} := S_c \cap \Z_{>k}.\]
		We then set $
		\lambda_c = \sum_{n \in S_c} a_np^n.$
		If $\lambda_c$ is nonzero and rational, we may replace $\lambda$ with $\lambda-\lambda_c$. In particular, we may assume each $\lambda_c$ is either
		zero
		or non-rational. Let $c_0$ be the largest number such that $\lambda_{c_0} \neq 0$. As
		$\lambda_{c_0}$ is irrational for each $k,m>0$ we have
		\begin{align}\label{eq: algebraic criterion for power series}
		m + S_{c_0}^k &\neq S_{c_0}^{k + m},
		\end{align}
		as otherwise the digits of $\lambda_{c_0}$ would eventually be periodic.
		
		Next, write $(1+x)^\lambda = \sum_{n=0} b_nx^n$. By using the product expansion
		$(1+x)^\lambda = \prod (1+x^{p^n})^{a_n}$ we see that
		\begin{align*}
		b_{c_0p^k}& \begin{cases}
		1 & k \in S_{c_0} \\
		0 & \text{otherwise}
		\end{cases}.
		\end{align*} 
		Consider the sequences $B_k=(b_{np^k})_{n \geq 1}$. By Theorem \ref{t: algebraic criterion via automata}
		if we prove that the $B_k$ are all distinct, we will know $(1+x)^\lambda$ is transcendental.  Note that the $c_0p^n$-th element in $B_k$ is $1$ if $n+k \in S_{c_0}^{k}$ and zero otherwise. In particular, if $B_{k}= B_{k+m}$, then we must have
		$m + S_{c_0}^k  = S_{c_0}^{k+m}$. However, by \eqref{eq: algebraic criterion for power series} we know this can never happen, which proves the lemma.
	\end{proof}

	\subsection{A short proof of the transcendence of $\AHp(x)$}
	The first step is to prove that $\AHp(x)$ satisfies a functional equation.
	\begin{lemma}
		\label{l: functional equation}
		For $\xi\in \F_p^\times$ we have $\AHp(\xi x)=\AHp(x)^{[\xi]}$, where $[\xi]\in \Z_p$
		denotes the Teichmuller lift of $\xi$. 
	\end{lemma}
	\begin{proof}
		For $f_0 \in \Q_p\powerseries{x}$ and $f_1,\dots,f_n \in x\Q_p\powerseries{x}$ we may take the composition $f_0 \circ \dots \circ f_n \in \Q_p\powerseries{x}$ and this operation is associative. For $c \in \Z_p$ we define
		\[g(x) = \sum_{n=1}^\infty \frac{x^{p^n}}{p^n}, ~~~L_c(x)=cx, ~~\text{and} ~~P_c(x)=(1+x)^c = \sum_{n=0}^\infty \binom{c}{n} x^n. \]
		Note that $E_p(x) = \exp\circ g$. Also, for $\xi \in \F_p^\times$ we have the relations:
		\[g\circ L_{[\xi]} = L_{[\xi]} \circ g ~~\text{and}~~\exp \circ L_{[\xi]} = P_{[\xi]} \circ (\exp(x) - 1).\]
		The first relation comes from the fact that $[\xi]^p=[\xi]$.
		Thus, we compute
		\begin{align*}
		E_p([\xi]x) &= \exp \circ g\circ L_{[\xi]}  \\
		&= \exp \circ L_{[\xi]} \circ g \\
		&= P_{[\xi]} \circ (\exp(x) - 1) \circ g \\
		&= P_{[\xi]} \circ (E_p(x)-1) = \AHp(x)^{[\xi]}.
		\end{align*}
		The lemma follows by reducing modulo $p$.
	\end{proof}
		
	\begin{theorem}
		The series $\AHp(x)$ is transcendental over $\F_p[x]$ if $p\geq 5$. 
	\end{theorem}
	\begin{proof}
		Take $c \in \F_p^\times$. Since $p\geq 5$ we can pick $c$ so that $[c]$ is not an integer and thus not rational. Assume $\AHp(x)$ is algebraic. This implies $\AHp(cx)$ is algebraic. Thus, by Lemma \ref{l: functional equation} we know that $\AHp(x)^{[c]}$ is algebraic. However, we know $[c]$ is not rational, so Lemma \ref{l: transcence of $(1+x)^c$} tells us that  $\AHp(x)^{[c]}$ is transcendental. This is a contradiction. 
	\end{proof}

	\begin{remark}
		One can get algebraic independence results using
		the full result of \cite{Allouche-MendesFrance-vanderPoorten}. In particular, the argument above
		can be modified to show $\dim_{\F_p(x)}(\AHp(cx))_{c \in \F_p^\times} \geq \phi(p-1)-1$. As we prove in \ref{s: proof of alg ind results} that $\dim_{\F_p(x)}(\AHp(cx))_{c \in \F_p^\times}=\phi(p-1)$, we omit the proof of this weaker result.
	\end{remark}

	\section{Solutions to iterated Frobenius equations} \label{s: solutions to iterated frobenius equations}
	Let $\F_q$ be the field with $q=p^f$ elements and fix $\zeta \in \F_q$ such that $\zeta,\zeta^p, \dots, \zeta^{p^{f-1}}$ is a basis
	of $\F_q$ over $\F_p$. Let $\Z_q$ denote the $p$-typical Witt vectors $W(\F_q)$ of $\F_q$. 
	By abuse of notation we regard the $\zeta$ as an element of $\Z_q$ by identifying $\zeta$ with its Teichmuller lift. It will be convenient to define $x_j^i = \zeta^{p^j}x^i$ and the set
	\begin{align*}
	X:=\left\{ x_j^i ~\middle|~ \begin{array}{c}1 \leq j \leq f \\ i \geq 1, p\nmid i
	\end{array}\right\}.
	\end{align*}

	\subsection{Iterated Frobenius equations}
		Let $\sigma: \Z_q\LS{x} \to \Z_q\LS{x}$ denote the $\Z_p$-linear endomorphism sending $x$ to $x^p$ and acting as the Frobenius endomorphism on $\Z_q$.
		Let $f \in x\Z_q\powerseries{x}$. We define $\tau (f)$ to be the unique solution to $y^\sigma - y = f$ in $x\Z_q\LS{x}$. Concretely, we have
		\[\tau(f) = \sum_{n=0}^\infty f^{\sigma^n}.\]
		We regard $\tau$ as a $\Z_p$-linear map. Note that $\tau$ satisfies these two fundamental equations:
		\begin{align}
			\tau(f)\tau(g) &= \tau(f\tau(g)) + \tau(g\tau(f)) - \tau(fg),\label{eq: tau multiplication}\\
			\tau(f^\sigma \tau(g)) &=\tau(f\tau(g)) + \tau(f^\sigma g) - f \tau(g) \label{eq: tau p power}.
		\end{align}
		\begin{definition}
			Let $[k]$ denote the set $\{1,\dots,k\}$ and let $[0]$ denote the empty set. An \emph{iterated Frobenius datum} $b$ (abbreviated IF-datum) of depth $k$ is
			a function 
			\begin{equation*}
				b:[k] \to x\Z_q\powerseries{x}.
			\end{equation*}
			If the image of $b$ is in $X$ we say $b$ is a $p$-primary $\IF$-datum.
			We denote the depth of $b$ by $\dep{b}$. We remark that there is exactly one
			$\IF$-datum of depth $0$. We define $D$ to be the set of all $\IF$-data. 
			\end{definition}
			\begin{definition}
				
		We let $b^{tr}$ denote
			the restriction of $b$ to $[k-1]$ (we think of $b$ as a list of elements and $b^{tr}$ as
			the truncation of the last element, hence the $tr$ in the superscript). Also, for any
			$z \in x\Z_q\powerseries{x}$ we let $\cat(b,z)$ to be the $\IF$-datum of depth $k+1$ defined by:
			\begin{align*}
				\cat(b,z)(i) &= \begin{cases} b(i) & i \in [k] \\
				z & i=k+1 
				\end{cases}.
			\end{align*}
		\end{definition}
	
		\begin{definition}
			Let $b$ be an IF-datum of depth $k$. We define
			$w_b\in x\Z_q\powerseries{x}$ by the recursive definition:
			\begin{align*}
				w_b &= \begin{cases}
					1 & \dep{b}=0 \\
					\tau(b(k)w_{b^{tr}}) & \dep{b}=k>0
				\end{cases}.
			\end{align*}
			We refer to $w_b$ as an iterative Frobenius solution (abbreviated $\IF$-solution).
		\end{definition}
		For an $\IF$-datum $b$ we may consider the iterative Frobenius equation:
		\[y_1^\sigma - y_1 = b(1), ~~ y_2^\sigma - y_2 = b(2)y_1, ~~  \dots,~~y_k^\sigma - y_k=b(k)y_{k-1}.\]
		There is a unique solution vector $(z_1,\dots,z_k) \in \F_q\powerseries{x}$ such that
		$z_i \in x^i\F_q\powerseries{x}$. Then $w_b$ is $z_k$.

		\subsection{Spaces generated by $\IF$-solutions}
		We make the following definitions:
		\begin{align*}
		M_k &:= \left \{\sum_{\substack{b \in D \\ \dep{b}\leq k}} a_b w_b ~|~ a_b \in \Z_q(x) \right\},\\
		\widehat{M}_k&:= \text{$p$-adic completion of $M_k $ in $\Z_q\powerseries{x}\langle x^{-1} \rangle $},\\
		A_k &:=\left \{\sum_{\substack{b \in D \\ \dep{b}\leq k}} a_b w_b ~|~ a_b \in \Z_q[x] \right\}.
		\end{align*}
		
		\begin{lemma}
			\label{c: tau applied to the A's}
			For $k\geq 0$ we have $\tau(xA_k) \subset A_{k+1}$. 
		\end{lemma}
		\begin{proof}
			 We proceed by induction on $k$. For $k=0$ we are reduced to showing $\tau(x_j^i) \in A_1$. 
			 Write $i=p^th$ where $h$ is coprime to $p$. Then 
			 \[ \tau(x_j^i) = \tau(x_{j-t}^{i-t}) - \sum_{m=1}^{t} x_{j-m}^{p^{t-m}h},\]
			 so $\tau(x_j^i) \in A_1$. 
			For $k>0$ it suffices to show $\tau(x_j^i w_b) \in A_{k+1}$ for any $w_b\in W$ with $\dep{w}=k$ and $k\geq 1$.
			Again we write $i=p^th$. 
			By repeatedly using \eqref{eq: tau p power}
			we obtain
			\[ \tau(x_j^i w_b) = \tau(x_{j-t}^{h}w_b) + \sum_{m=1}^{t} \tau(x_{j-m}^{p^{t-m}h}b(k) w_{b^{tr}} ) - \sum_{m=1}^t x_{j-m}^{p^{t-m}h}w_b. \]
			Note that $\tau(x_{j-t}^{h}w_b)$ is an $\IF$-solution of degree $k+1$, so it is
			contained in $A_{k+1}$. The second summation is clearly contained in $A_k$. Finally, 
			by our inductive hypothesis $\tau(x_{j-m}^{p^{t-m}h}b(k) w_{b^{tr}} )$ is contained in $A_k$.

		\end{proof}
		\begin{corollary}
			\label{c: IF-solutions are in M_n}
			For any $\IF$-datum $b$ of depth $k$ we have $w_b \in M_k$. 
		\end{corollary}
		\begin{lemma}
			\label{l: sigma on M_n}
			We have $\sigma(\widehat{M}_k) \subset \widehat{M}_k$.
		\end{lemma}
		\begin{proof}
			It suffices to prove $w_b^\sigma \in M_k$ for any $\IF$-datum with $\dep{b}=k$. We have
			$w_b^\sigma = w_b + b(k)w_{b^{tr}}$, which is in $M_k$. 
		\end{proof}
		
		\begin{lemma}
			\label{l: differentiating M_n}
			We have $\frac{d}{dx}\widehat{M}_k \subset \widehat{M}_k$.
		\end{lemma}
		\begin{proof}
			We induct on $k$. When $k=0$ the statement is immediate, since $\widehat{M}_0$ is
			the ring $\Z_q\langle x,x^{-1} \rangle$, which is preserved by $\frac{d}{dx}$.
			From the definitions and the Leibnitz rule, we reduce to showing
			$\frac{d}{dx}w_b \in \widehat{M}_k$ whenever $\dep{b}=k$. 
			Applying $\frac{d}{dx}$ to the equation $w_b^\sigma - w_b = b(k)w_{b^{tr}}$ and using the chain rule gives
			\[px^{p-1}(w_b')^\sigma  + w_b' = b(k)'w_{b^{tr}} + b(k)w_{b^{tr}}' .\]
			We deduce that 
			\begin{align}\label{eq l: differentiating M_k}
				w_b' &=  \cdot \sum_{n=0}^\infty p^nx^{p^{n+1}-1}( b(k)'w_{b^{tr}} + b(k)w_{b^{tr}}')^{\sigma^n}.
			\end{align}
			By our inductive hypothesis we know $w_{b^{tr}}' \in \widehat{M}_{k-1}$. Then from Lemma \ref{l: sigma on M_n} we see that each summand 
			 in \eqref{eq l: differentiating M_k} is in $\widehat{M}_{k-1}$. As the terms in the sum converge $p$-adically, the lemma follows.  
		\end{proof}
		\begin{lemma}
			\label{l: multiplying the Mns} 
				 For $k_1,k_2 \geq 0$ we have $\widehat{M}_{k_1}\widehat{M}_{k_2} \subset \widehat{M}_{k_1+k_2}$.
		\end{lemma}
		\begin{proof}
			We will show $A_{k_1}A_{k_2} \subset A_{k_1+k_2}$. The lemma will follow by inverting $x$ and taking the $p$-adic completion.
			We proceed by induction on $k=k_1+k_2$. The case $k=0$ is immediate. Let $k>0$ and assume the result holds for $k-1$. Let $b_i$ be an $\IF$-datum of degree $k_i$. Then using \eqref{eq: tau multiplication}
			we have 
			\begin{align*}
				w_{b_1}w_{b_2} &=\tau(b_1(k_1)w_{b_1^{tr}}w_{b_2}) + \tau(b_2(k_2)w_{b_2^{tr}}w_{b_1}) - \tau(b_1(k_1)w_{b_1^{tr}}\tau(b_2(k_2)w_{b_2^{tr}}).
			\end{align*}
			By lemma follows from our inductive hypothesis and Lemma \ref{c: tau applied to the A's}.
		\end{proof}
	
		\subsection{$\IF$-solutions modulo $p$.}
		\label{ss: spaces of iterated Frobenius equations modulo p}
		We now study the space of $\IF$-solutions modulo $p$. 
		\begin{definition}
			Let $b$ be an $\IF$-datum. We define $\overline{w}_b$ to be the
			reduction of $w_b$ modulo $p$ in $\F_q\powerseries{x}$. 
		\end{definition}
		We define the spaces
		\begin{align*}
			\overline{M}_k &:= \left\{ \sum_{\substack{b \in D \\ \dep{b}\leq k}} a_b \overline{w}_b~|~ a_b \in \F_q(x) \right\} \\
			\overline{M}&= \bigcup_{n\geq 1} \overline{M}_k.
		\end{align*}
		Our main result of this section is the following proposition.
		
		\begin{proposition}
			\label{p: linear independence of the w's mod p}
			The elements $\overline{w}_b$, where $b$ ranges over $b \in D$, are linearly independent
			over $\F_q(x)$. In particular, the set $\{\overline{w}_b\}_{b \in D}$ is a basis of $\overline{M}$ over $\F_q(x)$.
		\end{proposition}
		\begin{proof}
			It will be useful for us to introduce an ordering on $D$. We set
			$b_1>b_2$ if $\dep{b_1}>\dep{b_2}$ and for $\IF$-data with
			the same depth $k$ we use the lexicographical ordering on $(b(k), \dots, b(1))$.
			Consider $c=\{c_b\}_{b\in B}$ with $c_b \in \F_q(x)$, where almost all
			the $c_b=0$. We define $\max(c)$ to be the largest $b$ such that $c_b \neq 0$. 
			Assume there exists such a $c$ such that
			\begin{align}\label{eq: linear independence eq}
				\sum_{b \in B} c_b \overline{w}_b &= 0.
			\end{align}
			We may further assume that $\max(c)$ is minimal in the sense that for any other $c'=\{c_b'\}_{b \in B}$ where $\max(c')<\max(c)$ the corresponding sum $\sum c_b' \overline{w}_b$ is nonzero. Set $b_{max}=\max(c)$ and $k_{max}=\dep{b_{max}}$. By multiplying \eqref{eq: linear independence eq} with a constant we may assume $c_{b_{max}}=1$. 
			
			Let $F$ be the $p$-th power Frobenius. For $b\in D$ with $\dep{b}=k$ we have
			\[(F-1)c_b\overline{w}_b = c_b^p\overline{w}_b^p - c_b\overline{w}_b = (c_b^p-c_b)\overline{w}_b + c_b^pb(k)\overline{w}_{b^{tr}}.\]
			Thus, if we $F-1$
			to \eqref{eq: linear independence eq} we get
			\begin{align*}
				\sum_{b \in D} d_b \overline{w}_b &= 0,
			\end{align*} 
			where for $b$ of depth $k$ we have
			\begin{align*}
			d_b&:= (c_b^p-c_b)  + \sum_{\substack{b_0 \in D \\ b_0^{tr} = b}} c_{b_0}b_0(k+1).
			\end{align*}
			It is clear that $d_{b_{max}}=0$ and $\max(d)<\max(c)$. Thus, we are reduced
			to show that not all the $d_b$'s are zero, as this will contradict our minimality condition on $c$.
			
			First assume each $c_b \in \F_p$. Then we have
			\begin{align*}
				d_{b_{max}^{tr}} &= \sum_{\substack{b \in D \\ b^{tr} = b_{max}^{tr} }} c_{b}b(k_{max}),
			\end{align*}
			where the sum on the right is over all $b \in D$ that have the same truncation as $b^{tr}$.
			The terms $b_0(k_{max})$ are all distinct and of the form $x_j^i$. In particular, they are linearly independent over $\F_p$. Not all of the $c_b$'s are equal to zero, since we know $c_{b_{max}}\neq 0$, so we see that $d_{b_{max}^{tr}} \neq 0$. Next, assume that
			not all the $c_b$'s are in $\F_p$. Let $b_0$ be the largest $\IF$-datum such that
			$c_{b_0} \not \in \F_p$ and let $k_0=\dep{b_0}$. If $d_{b_0}=0$ we have
			\begin{align*}
				c_{b_0}^p-c_{b_0} = - \sum_{\substack{b \in D \\ b^{tr} = b_0}} c_{b}b(k_{0}+1). 
			\end{align*}
			We know that the left side is nonzero. In particular, the right side cannot be zero. However, the right side is an $\F_p$-linear combination of terms of the form $x_j^i=\zeta^{p^j}x^i$, where $p\nmid i$. In particular, the right side is a polynomial $f(x) \in \F_q[x]$ whose degree is coprime to $p$. However, for such an $f$, the equation $y^p-y=f(x)$ cannot have a solution in $\F_q(x)$, which gives a contradiction. 
		\end{proof}
		
		\begin{corollary}
			\label{c: the space M has no weird congruences}
			For all $k,m\geq0$ we have $p^m\Z_q\LS{x} \cap \widehat{M}_k = p^m\widehat{M}_k$.
			In particular, let $x,y \in \widehat{M}_{k+1}$. If $p^m x \equiv p^m y \mod \widehat{M}_k$
			then $x \equiv y \mod \widehat{M}_k$. 
		\end{corollary}
		
		In light of Proposition \ref{p: linear independence of the w's mod p} we make the following definitions.
		\begin{definition}
			Let $f \in \overline{M}$ and write $f = \sum a_b \overline{b}$, where almost all $a_b=0$. We define the depth
			of $f$, written as $\dep{f}$, to be the largest $k$ such that for some $b$ with $\dep{b}=k$ we have $a_b \neq 0$. In particular the depth of $f \in \overline{M}$ is the unique $k\geq 0$ such that $f \in \overline{M}_k$ and $f \not \in \overline{M}_{k-1}$. 
		\end{definition}
		
		\begin{corollary}\label{c: linear independence for different depths}
			Let $\alpha_1,\alpha_2, \dots$ be a sequence of elements in $\overline{M}$ with $\dep{\alpha_k}=k$ for each $k$. Then
			the $\alpha_k$'s are linearly independent over $\F_p(x)$.
		\end{corollary}
		
		\begin{definition}\label{def: dual basis}
			Let $\{\dot{w}_b\}_{b \in D}$ be the basis of $\Hom_{\F_p\LS{x}}(\overline{M},\F_p\LS{x})$ that is dual to 
			$\{\overline{w}_b\}_{b \in D}$. That is, for $b,b' \in D$ we have $\dot{w}_b(\overline{w}_{b'})$ is one if $w=w'$ and
			zero otherwise.			
		\end{definition}
	
		\subsection{Products of $\IF$-solutions} \label{ss: Products of solutions to recursive Frobenius equations}for the the product $w_{b_1}\dots w_{b_r}$
		in terms of the $\IF$-data $b_1,\dots,b_r$.

		\begin{definition} \label{def: merging function}
			
			Let $k_1,\dots,k_r\geq 1$ and set $k=k_1+\dots+k_r$. We define
			\begin{align*}
				[k_1,\dots,k_r] := \left\{(i,j) \middle| \begin{array}{c}1 \leq i \leq r \\ j \in [k_i]
				\end{array}\right\}.
			\end{align*}
			A splicing function is a bijection $\eta: [k] \to [k_1,\dots,k_r]$ such that 
			for any $1 \leq j_1<j_2 \leq k_i$ we have $\eta^{-1}(i,j_1)<\eta^{-1}(i,j_2)$. 
			We let $C(k_1,\dots,k_r)$ denote the set of all splicing functions. Note that
			$C(k_1,\dots,k_r)$ has cardinality $\binom{k}{k_1,\dots,k_r}$.
		\end{definition}
	
		\begin{definition} \label{def: splicing IF datum}
			Let $b_1,\dots,b_r$ be $\IF$-data with $\dep{b_i}=k_i$. For $\eta \in S(k_1,\dots,k_r)$
			we define the splicing of $b_1,\dots,b_r$ according to $\eta$ to be the depth $k$ $\IF$-datum defined by the composition
			\[b_{\eta}: [k] \xrightarrow{\eta} [k_1,\dots,k_r] \xrightarrow{(i,j) \mapsto b_i(j)} x\F_q\powerseries{x}.\]
			We let $S(b_1,\dots,b_r)$ denote the multiset of all possible splicings of $b_1,\dots,b_r$. That is,
			\begin{align*}
				S(b_1,\dots,b_r) := \left\{b_{\eta}~\middle|~ \eta \in C(k_1,\dots,k_r) \right\}.
			\end{align*}
			\end{definition}
		The following lemma follows immediately from these definitions.
		\begin{lemma}
			\label{l: merging set}
			Let $b_i$ and $k_i$ be as in Definition \ref{def: splicing IF datum}. Then
			\begin{align*}
				S(b_1,\dots,b_r) &=\bigsqcup_{i=1}^r \left\{ \cat(b_\eta, b_i(k_i))~\middle | ~ \eta \in S(b_1,\dots,b_{i-1},b_i^{tr}, b_{i+1}, \dots, b_r) \right\}.
			\end{align*}
		\end{lemma}
	
		\begin{lemma}\label{l: preliminary multiplying lemma}
			Continuing with the same notation, we define 
			\[ \mathcal{B}_r := \prod_{i=1}^r w_{b_i}~~ \text{ and } ~~\mathcal{B}_r^{(j)} := w_{b_j^{tr}} \prod_{\substack{1 \leq i \leq r \\ i \neq j}} w_{b_i}.\]
			Then we have
			\begin{align*}
				\mathcal{B}_r &=\sum_{i=1}^r \tau\left (b_i(k_i) \mathcal{B}_r^{(i)} \right) \mod \widehat{M}_{k-1}.
			\end{align*}
		\end{lemma}
		\begin{proof}
			We proceed by induction on $r$ with the base case $r=1$ being immedaite. Let $r>1$ and assume the lemma holds for $r-1$. In particular, we have
			\begin{align}\label{eq: ultra product prop equation 2}
			\mathcal{B}_{r-1} &\equiv \sum_{i=1}^{r-1} \tau\left (b_i(k_i) \mathcal{B}_{r-1}^{(i)} \right) \mod \widehat{M}_{k-k_r-1}.
			\end{align}
			We know $\widehat{M}_{k-k_r-1}\widehat{M}_{k_r} \subset \widehat{M}_{k-1}$ by Lemma \ref{l: multiplying the Mns}. Thus, multiplying \eqref{eq: ultra product prop equation 2} by $w_{b_r}$: 
			\begin{align}\label{eq: ultra product prop equation 5}
			\mathcal{B}_r &\equiv w_{b_r} \sum_{i=1}^{r-1} \tau\left (b_i(k_i) \mathcal{B}_{r-1}^{(i)} \right ) \mod \widehat{M}_{k-1}. 
			\end{align}
			For $1\leq i \leq r-1$ we know by \eqref{eq: tau multiplication} that
			\begin{align}\label{eq: ultra product prop equation 3}
				w_{b_r}\tau(b_i(k_i)\mathcal{B}_{r-1}^{(i)}) &=\tau\left(b_r(k_r)w_{b_r^{tr}} \tau\left (b_i(k_i)\mathcal{B}_{r-1}^{(i)}\right)\right)  + \tau\left(b_i(k_i)\mathcal{B}_{r-1}^{(i)} w_{b_r}\right) - \tau\left(b_r(k_r)w_{b_r^{tr}} b_i(k_i)\mathcal{B}_{r-1}^{(i)}\right).
			\end{align}
			We know $\mathcal{B}_{r-1}^{(i)} \in \widehat{M}_{k-k_r -1}$ and $w_{b_r^{tr}} \in \widehat{M}_{k_r-1}$, so that the last term in the right side of \eqref{eq: ultra product prop equation 3} is in $\widehat{M}_{k-1}$. As
			$\mathcal{B}_{r-1}^{(i)} w_{b_r}=\mathcal{B}_r^{(i)}$ we have
			\begin{align}\label{eq: ultra product prop equation 4}
				w_{b_r}\tau(b_i(k_i)\mathcal{B}_{r-1}^{(i)}) &\equiv  \tau\left(b_r(k_r)w_{b_r^{tr}} \tau\left (b_i(k_i)\mathcal{B}_{r-1}^{(i)}\right)\right)  + \tau\left(b_i(k_i)\mathcal{B}_{r}^{(i)}\right) \mod \widehat{M}_{k-1}.
			\end{align}
			Next, by our inductive hypothesis we know 
			\begin{align*}
				\mathcal{B}_{r-1} &\equiv \sum_{i=1}^{r-1} \tau\left (b_i(k_i) \mathcal{B}_{r-1}^{(i)} \right)\mod\widehat{M}_{k-k_r-1}.
			\end{align*}
			We can combine this with \eqref{eq: ultra product prop equation 4} and \eqref{eq: ultra product prop equation 5} to get
			\begin{align*}
				\mathcal{B}_r &\equiv  \tau\left(b_r(k_r)w_{b_r^{tr}} \mathcal{B}_{r-1}\right) + \sum_{i=1}^{r-1} \tau\left(b_i(k_i)\mathcal{B}_{r}^{(i)}\right)\mod\widehat{M}_{k-k_r-1}.
			\end{align*}
			The result follows by observing that $\mathcal{B}_{r}^{(r)}=w_{b_r^{tr}} \mathcal{B}_{r-1}$. 
		\end{proof}
	
		\begin{proposition}
			\label{p: most general multiplying w lemma} Continue with the notation from Lemma \ref{l: preliminary multiplying lemma}.				
				We have
			\begin{align*}
				\mathcal{B}_r &\equiv \sum_{c \in S(b_1,\dots,b_r)} w_c \mod \widehat{M}_{k-1}.
			\end{align*}
		\end{proposition}
		\begin{proof}
			We induct on $r$. When $r=1$ the result is immediate. Let $r>1$ and assume the result holds for $r-1$. From Lemma \ref{l: preliminary multiplying lemma}
			we know 
			\begin{align*}
			\mathcal{B}_r &=\sum_{i=1}^r \tau\left (b_i(k_i) \mathcal{B}_r^{(i)} \right) \mod \widehat{M}_{k-1}.
			\end{align*}
			By our inductive hypothesis we have
			\begin{align*}
				\mathcal{B}_r^{(i)} &\equiv \sum_{c \in S(b_1,\dots,b_{i-1},b_i^{tr}, b_{i+1}, \dots, b_r)} w_c \mod \widehat{M}_{k-2}.
			\end{align*}
			The proposition then follows from Lemma \ref{l: merging set} and Corollary \ref{c: tau applied to the A's}
		\end{proof}
	
		We will now give some corollaries of Proposition \ref{p: most general multiplying w lemma}. First
		we need a definition.
		\begin{definition}
			Let $k>0$ and $z \in x\F_q\powerseries{x}$. Let $\beta_z^k$ to be the $\IF$-datum of
			depth $k$ such that $\beta_z^k(i)=z$ for each $1\leq i \leq k$. We define $\omega_z^k$ to be $w_{\beta_z^k}$. We let $\overline{\omega}_z^k$ be the reduction of $\omega_z^k$
			modulo $p$ and we let $\dot{\omega}_z^k$ be $\dot{ w_{\beta_z^k}}$ as defined in Definition \ref{def: dual basis}.
		\end{definition}
		We readily compute that
		\begin{align*}
			\omega^k_z &= \sum_{i_k\geq \dots \geq i_1 \geq 0} z^{bp^{i_1 + \dots + i_k}}.
		\end{align*}
		The elements $\omega_z^k$ and $\overline{\omega}^k_z$ will appear when studying the higher derivatives of
		the Artin-Hasse exponential. 
		We have the following corollaries of Proposition \ref{p: most general multiplying w lemma}.
		
		\begin{corollary} \label{c: multiplying the wks}
			Fix $z \in \F_q\powerseries{x}$. Let $k_1,\dots,k_r\geq 1$ and set $k=k_1 + \dots + k_r$. Then
			\begin{align*}
			\prod_{i=1}^r {\omega}^{k_1}_z &\equiv  \binom{k}{k_1,\dots,k_r} {\omega}^{k}_z \mod \widehat{M}_{k-1}.
			\end{align*}
		\end{corollary}
		
		\begin{proof}
			This follows from Proposition \ref{p: most general multiplying w lemma} and the fact that there are $\binom{k}{k_1,\dots,k_r}$ splicing functions.
		\end{proof}
		
		\begin{corollary}
			\label{l: approximating $tau(x)^n$}
			We have $\tau(z)^k \equiv k!{\omega}^{k}_z \mod \widehat{M}_{k-1}$.
		\end{corollary}
		
		\begin{corollary}
			\label{c: need correct letters for multiplying}
			Let $b_1,\dots,b_r \in D$, let $k_i=\dep{b_i}$, and set $k=k_1+ \dots+ k_r$. 
			Then for any $z \in X$ we have
			\begin{align*}
				\dot{\omega}_z^k(\overline{w}_{b_1}\dots \overline{w}_{b_r}) &= \begin{cases}
					\binom{k}{k_1,\dots,k_r} & \text{ each }\overline{w}_i=\overline{\omega}_{z}^{k_i} \\
					0 & \text{ otherwise} 
				\end{cases}.
			\end{align*}
		\end{corollary}

		\section{Higher derivatives and the Artin-Hasse exponential}
		\label{s: higher derivatives and the artin-hasse exponential}
		
		\subsection{Higher derivatives and a transcendental criterion}
		
		We define the \emph{higher derivatives} on $\F_q\LS{x}$ and $\Z_q\LS{x}$:
		\[\ipartial{k} = \frac{1}{k!} \frac{d^k}{d^kx} .\]
		We define the \emph{higher logarithmic derivatives} of $f\in \F_q\LS{x}$ or  $f \in \Z_q\LS{x}$ by 
		\[\dlog{k}(f) = \frac{\ipartial{k}(f)}{f}.\]
		
		Note that $\ipartial{k}$ and $\dlog{k}$ restrict to maps on $\F_p(x)$. Furthermore, if $K$ is a finite separable extension
		of $\F_p(x)$, then each $\ipartial{k}$ and $\dlog{k}$ extends uniquely to $K$ by a theorem of Schmidt (see \cite[\S 2]{Matzat-vanderPut-Iterative_diff_eqs}). 
		We begin with the following general transcendence result. 
		
		\begin{theorem}\label{t: iterative differential equation criterion for transcendence}
			Let $\alpha_1,\alpha_2,\dots$ be a sequence of power series in $\F_p\llbracket x \rrbracket$. Let $f(x) \in \F_p\llbracket x \rrbracket$ be a power series satisfying the iterative differential equation
			\[ \dlog{k}(f) = \alpha_k.\] 
			If $\F_p(x)[\alpha_i]_{i\geq 1}$ is an infinite extension of $\F_p(x)$, then $f(x)$ is transcendental over $\F_p[x]$. 
		\end{theorem}
		
		\begin{proof}
			Assume $f(x)$ is algebraic (necessarily separable) and contained in a field $K$. 
			As $\dlog{k}$ extends to $K$, we see that $\alpha_n $ is contained in $K$ as well. This contradicts our assumption that $\F_p(x)(\alpha_1,\alpha_2,\dots)$ is an infinite extension of $\F_p(x)$.
		\end{proof}
		
		\begin{corollary}
			\label{c: diff criterion using depths}
			Let $\alpha_1,\alpha_2,\dots$ be a sequence of power series contained in $\overline{M}$ with $\dep{\alpha_k}=k$ for each $k$.
			Let $f(x) \in \F_p\llbracket x \rrbracket$ be a power series satisfying the iterative differential equation
			\[ \dlog{k}(f) = \alpha_k.\] 
			Then $f(x)$ is transcendental.
		\end{corollary}
		\begin{proof}
			From Corollary \ref{c: linear independence for different depths} we know that the $\alpha_k$'s are linearly independent over
			$\F_p(x)$, so they cannot be contained in a finite extension. Then use Theorem \ref{t: iterative differential equation criterion for transcendence}.
		\end{proof}

		\subsection{Estimating the higher derivatives of $E_p(x)$.}
		Let $L$ denote the differential operator $\frac{d}{dx} - \frac{\tau(x)}{x}$. 
		Using the description $
		E_p(x) = \exp\left( \sum_{n=0}^\infty \frac{x^{p^n}}{p^n} \right )$
		we see that for $k\geq 1$ we have
		\begin{align}\label{eq: higher derivative formula}
		\dlog{k}(E_p(x)) &= L^{k-1}\Big( \frac{\tau(x)}{x} \Big ).
		\end{align}  
		We have the following proposition.
		\begin{proposition}
			\label{p: first estimate of higher log derivative}
			For $k\geq 1$ we have
				\begin{align*}
					\dlog{k}(E_p(x)) \equiv \frac{\tau(x)^k}{x^k} \mod \widehat{M}_{k-1}.
				\end{align*}
			
		\end{proposition}
		\begin{proof}
			We proceed by induction on $k$. When $k=1$ the result is immediate. Assume
			the result holds for $k$. Then we have
			\begin{align*}
				L^{k+1}\Big( \frac{\tau(x)}{x} \Big )= L\Big( \frac{\tau(x)^k}{x^k} + c \Big ),
			\end{align*}
			where $c \in \widehat{M}_{k-1}$. We know $\frac{d}{dx}(\frac{\tau(x)^k}{x^k} + c) \in \widehat{M}_{k}$ from Lemma \ref{l: differentiating M_n}.
			We also know $\frac{\tau(x)}{x}c \in \widehat{M}_k$ by 
			Lemma \ref{l: multiplying the Mns}. Thus,
			\begin{align*}
				L\Big( \frac{\tau(x)^k}{x^k} + c \Big ) &\equiv \frac{\tau(x)^{k+1}}{x^{k+1}} \mod \widehat{M}_{k}.
			\end{align*}
			The Proposition follows from \eqref{eq: higher derivative formula}.			
		\end{proof}

		\begin{corollary}
			\label{c: estimate of higher derivative of $E_p(x)$}
			For $k\geq 1$ we have
			\begin{align*}
			\dlog{k}(E_p(x)) \equiv \frac{\omega_x^k}{x^k} \mod \widehat{M}_{k-1},\\
			\dlog{k}(\AHp(x)) \equiv \frac{\overline{\omega}^k_x}{x^k} \mod \overline{M}_{k-1}.
			\end{align*}
		\end{corollary}
		\begin{proof}
			From Corollary \ref{l: approximating $tau(x)^n$} and Propostion \ref{p: first estimate of higher log derivative} we know $\dlog{k}(E_p(x)) \equiv \frac{k!\omega_x^k}{x^k}\mod \widehat{M}_{k-1}$. The equation about $\dlog{k}(E_p(x))$ then follows from Corollary \ref{c: the space M has no weird congruences}.			
			The equation about $\dlog{k}(\AHp(x))$ comes by reducing modulo $p$. 
		\end{proof}
	
		At this point we may establish the transcendence of $\AHp(x)$.

		\begin{theorem}
			\label{t: AHp is transcendental}
			The series $\AHp(x)$ is transcendental over $\F_p(x)$. 
		\end{theorem}
		\begin{proof}
			Set $\alpha_k=\dlog{k}(\AHp(x))$. From Corollary \ref{c: estimate of higher derivative of $E_p(x)$} we see that $\alpha_k \in \overline{M}$ and $\dep{\alpha_k}=k$. The theorem follows from Corollary \ref{c: diff criterion using depths}.			
		\end{proof}

		\subsection{Estimating $\dlog{k}\left[\AHp(f_1)^{n_1}\dots \AHp(f_r)^{n_r}\right]$}
		Recall from the introduction that for $f=\sum c_ix^i  \in x\F_q[x]$ we define
		\[f^* =\sum_{\substack{i\geq 0\\ p \nmid i}} c_i^*(f) x^i,~~~ \text{ with}~~~
		c_i^*(f) = \sum_{k\geq 0} c_{ip^k}^{1/p^k}.\]
		We can express $c_i^*(f)$ uniquely as $\sum_{j=0}^{f-1} c_{x_j^i}^* \zeta^{p^j}$
		where $c_{x_j^i}^* \in \F_p$. In particular, we have 
		\begin{align*}
			f^* &= \sum_{z \in X} c_z^* z. 
		\end{align*}
		
		In this subsection we prove the following proposition.
		\begin{proposition}
			\label{p: estimate of dlog for products}
			Let $k\geq 0$ and let $n_1,\dots,n_r\geq 0$. Let $f_1,\dots,f_r \in x\F_q[x]$. Then $\dlog{k}\left[\AHp(f_1)^{n_1}\dots \AHp(f_r)^{n_r}\right]$ is in
			$\overline{M}_{k}$ and for any $z \in X$ we have
			\begin{align}
				\dot{\omega}^k_z\left(\dlog{k}\left[\AHp(f_1)^{n_1}\dots \AHp(f_r)^{n_r}\right]\right)&=\left( \sum_{i=1}^r n_i c_z^{*}(f_i)\dlog{1}(f_i)\right)^k. \label{eq: coefficient of wk(nu)}
			\end{align}
			In particular, $\dlog{k}\left[\AHp(f_1)^{n_1}\dots \AHp(f_r)^{n_r}\right]$ has depth exactly $k$ if
			there exists $z \in X$ such that $\displaystyle\sum_{i=1}^r n_i c_z^{*}(f_i)\dlog{1}(f_i)$ is nonzero.
		\end{proposition}
		
		The proof of Proposition \ref{p: estimate of dlog for products} is broken into several steps. The main ingredients are Corollary \ref{c: estimate of higher derivative of $E_p(x)$} 
		and the following well known identities on higher derivatives.
		\begin{align}
			\ipartial{k}(f_1\dots f_n) &= \sum_{\substack{k_1, \dots, k_n \geq 0\\ k_1 + \dots +k_n = k}} \ipartial{k_1}(f_1)\cdot \dots\cdot  \ipartial{k_n}(f_n) \label{eq: higher product rule},\\
			\ipartial{k}(f(g)) &= \sum_{j=1}^k \ipartial{j}(f)(g) \cdot \sum_{\substack{k_1,\dots,k_j\geq 1 \\ k_1 + \dots + k_j = k}} \ipartial{k_1}(g) \cdot \dots \cdot \ipartial{k_j}(g). \label{eq: chain rule}
		\end{align}

		\begin{lemma}
			\label{l: partials of powers of AH}
			Let $n$ be any $p$-adic integer. We have 
			\begin{align*}
				\dlog{k}(\AHp(x)^n) &\equiv  n^k \frac{\overline{\omega}_{x}^k}{x^k} \mod \overline{M}_{k-1}
			\end{align*}
		\end{lemma}
		\begin{proof}
			First assume $n$ is a positive integer. By Corollary \ref{c: estimate of higher derivative of $E_p(x)$}, Corollary \ref{c: multiplying the wks}, and \eqref{eq: higher product rule} we have
			\begin{align*}
				\dlog{k}(\AHp(x)^n) &\equiv \sum_{\substack{k_1, \dots, k_n \geq 0\\ k_1 + \dots +k_n = k}}  \frac{1}{x^k} \overline{\omega}^{k_1}_x \dots \overline{\omega}^{k_n}_x \mod \overline{M}_{k-1} \\
				&\equiv  \sum_{\substack{k_1, \dots, k_n \geq 0\\ k_1 + \dots +k_n = k}} \binom{k}{k_1,\dots,k_n} \frac{\overline{\omega}^{k}_x}{x^k} \mod  \overline{M}_{k-1} \\
				&\equiv n^k \frac{\overline{\omega}^{k}_x}{x^k} \mod  \overline{M}_{k-1}.
			\end{align*}
			For the general result, write $n = a + p^m b$. If $m > v_p(k)$ we have
			$\dlog{k}(\AHp(x)^n) = \dlog{k}(\AHp(x)^a)$, since $\ipartial{k}(f^{p^m})=0$.
			The general result follows.
		\end{proof}
	
		\begin{corollary}
			\label{c: partials powers of AH composed with f} 
			For $f \in x\F_p[x]$ we have 
			\begin{align*}
			\dlog{k}(\AHp(f)^n) &\equiv
			n^k \frac{\overline{\omega}_{x}^k(f)}{x^k} (\dlog{1}f)^k  \mod \overline{M}_{k-1} 
			\end{align*}
		\end{corollary}
		\begin{proof}
			From Lemma \ref{l: partials of powers of AH} we see that $\dlog{i}(\AHp(x)^n) \in \overline{M}_{k-1}$ for $1 \leq i <k$. Note that for $g \in \overline{M}_{k-1}$ and $h \in x\F_p[x]$ we have $g\circ h\in \overline{M}_{k-1}$, so that $\dlog{i}(\AHp(x)^n)\circ f\in \overline{M}_{k-1}$
			for $1 \leq i <k$. The Corollary follows from \eqref{eq: chain rule}
			and Lemma \ref{l: partials of powers of AH}.
		\end{proof}
		
		\begin{lemma} \label{l: plugging f in the sk}
			Let $f \in x\F_p[x]$. Then 
			\begin{align*}
				\overline{\omega}_x^k ( f) &\equiv \sum_{\substack{b \in D \\ \dep{b}=k}} c_b^{*}(f) \overline{w}_b \mod \overline{M}_{k-1}, \text{ where} \\
				c_b^{*}(f) &= \prod_{i=1}^k c^*_{b(i)}(f).
			\end{align*}
		\end{lemma}
	
		\begin{proof}
		 	First note that $\overline{\omega}_x^k ( f) \equiv \overline{\omega}_x^k ( f^{*})\mod \overline{M}_{k-1}$. This can be proven by inducting on $k$ and using
		 	\eqref{eq: tau p power}. We are therefore reduced to the case where $f=f^{*}$. 
		 	The result follows from the $\F_p$-linearity of $\tau$ modulo $p$ and unraveling the nested $\tau$'s.
		\end{proof}
	
		\begin{corollary}
			\label{c: coefficient of wbk}
			We have
			\begin{align*}
				\dot{\omega}_{z}^k\left(\overline{\omega}_x^k ( f) \right) &= c_z^*(f)^k.
			\end{align*}
		\end{corollary}
	
		\begin{lemma}
			\label{l: coefficient for products of sk(fi)}
			Let $k_1,\dots,k_r \geq 1$ with $k_1+\dots+k_r =k$. 
			Then
			\begin{align*}
				\dot{\omega}_{z}^k\left( \overline{\omega}_x^{k_1} ( f_1(x))\dots \overline{\omega}_x^{k_r} ( f_r(x)) \right ) = \binom{k}{k_1,\dots,k_r} \prod_{i=1}^r c_z^*(f_i)^{k_i}
			\end{align*}
		\end{lemma}
		\begin{proof}
			This follows from Corollary \ref{c: coefficient of wbk} and Corollary \ref{c: need correct letters for multiplying}.
		\end{proof}
	
			\begin{proof}[Of Proposition \ref{p: estimate of dlog for products}]
				From \eqref{eq: higher product rule} and Corollary \ref{c: partials powers of AH composed with f} we have
				\begin{align*}
					\dlog{k}\left[\AHp(f_1)^{n_1}\dots \AHp(f_r)^{n_r}\right] &= \sum_{\substack{k_1, \dots, k_r \geq 0\\ k_1 + \dots +k_r = k}} \prod_{i=1}^r \dlog{k_i}(\AHp(f_i)^{n_i}) \\
					& \equiv \sum_{\substack{k_1, \dots, k_r \geq 0\\ k_1 + \dots +k_r = k}} \prod_{i=1}^r n_i^{k_i}(\dlog{1}(f_i))^{k_i} \prod_{i=1}^r\overline{\omega}_x^{k_i}(f)\mod \overline{M}_{k-1}.
				\end{align*}
				Applying Lemma \ref{l: coefficient for products of sk(fi)} to each $\displaystyle \prod_{i=1}^r\overline{\omega}_x^{k_i}(f)$ gives
				\begin{align*}
					\dot{\omega}_z^{k}\left(\dlog{k}\left[\AHp(f_1)^{n_1}\dots \AHp(f_r)^{n_r}\right]\right) &=\sum_{\substack{k_1, \dots, k_r \geq 0\\ k_1 + \dots +k_r = k}} \binom{k}{k_1,\dots,k_r} \prod_{i=1}^s \left[ n_i\dlog{1}(f_i)c_z^*(f_i)\right]^{k_i} \\
					&= \left(\sum_{i=1}^r n_i\dlog{1}(f_i)c_z^*(f_i)\right)^k.\qedhere
				\end{align*}
			\end{proof}
		
			\section{Algebraic independence results}
			\label{s: proof of alg ind results}
			In this section we prove Theorem \ref{t: algebraic independence result} and Theorem \ref{t: automata result}. 
			\subsection{Iterative differential Galois theory and transcendence}\label{ss: ID Galois theory and transcendence}
			We will make use of iterative differential ($\ID$-for short) Galois theory, as expounded on by Matzat and van der Put in \cite{Matzat-vanderPut-Iterative_diff_eqs} (in particular \S 3 and \S 4 in this article.) Let $K$ be the algebraic closure
			of $\F_p(x)$ in $\overline{\F}_p\LS{x}$. The higher derivatives $\ipartial{k}$ make $K$ an $\ID$-field. For any $f \in \F_p[x]$ we recursively define a sequence $A_{f,0},A_{f,1}, \dots \in K$ by
			\[A_{f,0}=1 ~\text{ and } A_{f,k} = -\sum_{i=0}^{k-1}A_{f,i}\dlog{k-i}(\AHp(f)).\]
			We then define an $\ID$-module $N_f= Ke_f$ over $K$ of rank one by the rule:
			\[\ipartial{k}(e_f) = A_{f,k}e_f. \]
			Note that $N_f$ trivializes over the ring $K\left (\AHp(f),\frac{1}{\AHp(f)}\right)$.
			For $f_1,\dots,f_r \in x\overline{F}_p\LS{x}$ we define the rank $r$ $\ID$-module:
			\begin{align*}
				N:= \bigoplus_{i=1}^r N_{f_1}.
			\end{align*}
			Then $N$ has a full system of solutions over the ring
			\begin{align*}
				R_0:= K\left (\AHp(f_1), \dots, \AHp(f_r),\frac{1}{\AHp(f_1)\dots\AHp(f_r)}\right).
			\end{align*}
			The Picard-Vessiot ring of $N$ is $R=R_0/I$, where $I$ is a maximal $\ID$-ideal (i.e. $I$ is a maximal element in the ordered set of ideals of $R_0$ satisfying $\ipartial{k}(I) \subset I$.) The $\ID$-Galois group $\mathscr{G}(N)$ is the group of $\ID$-automorphisms of $R$ that fix $K$. A key fact is that $\mathscr{G}(N)$ is reduced algebraic group over the field of constants $\overline{\F}_p$ and that $\Spec(R)$ is an $\mathscr{G}(N)$-torsor over $K$. In particular we have
			\[\dim_K (\fracfield{R_0}) \geq \dim_K(\fracfield{R})=\dim_{\overline{\F}_p}(\mathscr{G}(N)).\]
			Thus, it suffices to show $\mathscr{G}(N)$ has dimension $r$. 
			
			On the other hand, there is a Tannakian interpretation of the $\ID$-Galois group. Let $\ID_{K}$ denote the category of $\ID$-modules over $K$. Then $\ID_K$ is a $\overline{\F}_p$-linear tensor category. For any object $N$ in $\ID_K$ we let $[N] $ be the full subcategory of $\ID_{K}$ generated by tensor powers of $N$ and its dual. Then $[N]$ is again a $\overline{\F}_p$-linear tensor category and thus isomorphic to the category
			of representations for an affine algebraic group $\Gal(N)$ defined over $\overline{\F}_p$. Then by standard arguments (see \cite{vanderPut-Singer-Galois_theory_lin_diff_eqs} for the characteristic $0$ case and see \cite{Matzat-vanderPut-Iterative_diff_eqs} for our precise situation) we have
			\begin{align*}
				\mathscr{G}(N) \cong \Gal(N).
			\end{align*}
			Note that since $N_{f_i}$ has rank one we know $\Gal(N_{f_i})$ is a closed subgroup of  $\mathbb{G}_m$. In particular, we see that $\Gal(N)$ is a closed subgroup of $\mathbb{G}_m^r$. To prove that $\Gal(N)\cong \mathbb{G}_m^r$ it is enough to show that for any $n_1,\dots,n_r \in \Z$ that are not all zero, the $\ID$-module $N_{f_1}^{\otimes n_1} \otimes \dots \otimes N_{f_r}^{\otimes n_r}$ is nontrivial. This amounts to proving $\AHp(f_1)^{n_1}\dots \AHp(f_r)^{n_r}$ is transcendental over $K$. Furthermore, we are easily reduced
			to the case where not all the $n_i$'s are divisible by $p$. 
			
			\subsection{Proof of Theorem \ref{t: algebraic independence result}}
			Let $f_1,\dots,f_r$ satisfy one of the two conditions from Theorem \ref{t: algebraic independence result}. From the discussion in \S \ref{ss: ID Galois theory and transcendence}
			we must show $\AHp(f_1)^{n_1}\dots \AHp(f_r)^{n_r}$ is transcendental, where we may assume that not all the $n_i$'s are divisible by $p$. 
			From Corollary \ref{c: diff criterion using depths} and Proposition \ref{p: estimate of dlog for products} we are reduced to showing that for some
			$z \in X$ we have
			\begin{align}\label{eq: thing we want to be nonzero}
				 \sum_{i=1}^r n_i c_z^{*}(f_i)\dlog{1}(f_i) &\neq 0.
			\end{align}

			\paragraph{The case where the $f_i^{*}$'s are not all zero and no non-trivial power product of the $f_i's$ are contained in $\F_p(x)^p$:}
			We have 
			\begin{align*}
				\sum_{i=1}^r n_i c_z^{*}(f_i)\dlog{1}(f_i) & = \dlog{1}\left[\prod_{i=1}^{r} f_i^{n_i c_z^{*}(f_i)} \right].
			\end{align*}
			By our assumption we know the product in the $\dlog{1}$ is not in $\F_p(x)^p$. The theorem follows by observing that
			$\ker(\dlog{1}) = (\F_p(x)^p)^\times$. 
			
			\paragraph{The case where $f_1^{*}, \dots, f_r^{*}$ are linearly independent:} 
			Without loss of generality assume $p \nmid n_1$. 
			We may also assume that none of the $f_i$'s are $p$-th powers. Let $g\in \F_p[x]$ be an irreducible polynomial dividing $f_1$ and let $v_g$ denote the valuation associated to $g$. Write $f_i=g^{b_i}h_i$, where $g \nmid h_i$.
			We can find $g$ so that $p\nmid b_1$, as otherwise $f_1$ would be a $p$-th power.
			Since $f_1^{*}, \dots, f_r^{*}$ are linearly dependent over $\F_p$,
			we know
			\begin{align*}
				n_1 b_1f_1^{*} + \dots + n_r b_rf_r^{*} &\neq 0.
			\end{align*}
			In particular, there exists $z$ coprime to $p$ such that
			\begin{align*}
				D_z:=n_1 b_1c_z^{*}(f_1)+ \dots + n_r b_r c_z^{*}(f_r)&\neq 0.
			\end{align*}
			The we have
			\begin{align*}
				\sum_{i=1}^r n_i c_z^{*}(f_i)\dlog{1}(f_i) &=D_z \dlog{1}(g) +\sum_{i=1}^r n_i c_z^{*}(f_i)\dlog{1}(h_i).
			\end{align*}
			If $z$ is a root of $g$, we see that $D_z \dlog{1}(g)$ has a simple pole at $z$, while $\dlog{1}(h_i)$ is regular at $z$.
			In particular, the right side of this sum has a simple pole at $z$ and thus is nonzero. 
			
			\subsection{Proof of Theorem \ref{t: automata result}}
			We now prove Theorem \ref{t: automata result}. 
			Let $\xi$ be a primitive $p-1$-th root of unity. Set $r=\phi(p-1)$. First, we show that
			$\AHp(x), \dots, \AHp(\xi^{r-1}x)$ are algebraically independent. From the 
			 $\ID$-Galois group argument from \S \ref{ss: ID Galois theory and transcendence} it suffices to show \[A(x):=\prod_{i=0}^{r-1} \AHp(\xi^{i}x)^{n_i}\]
			is transcendental, where $n_0, \dots, n_{r-1}$ are integers not all divisible by $p$. Indeed, this implies that the $\ID$-Galois group of the $\ID$-module $N_{x} \oplus N_{\xi x} \oplus \dots \oplus N_{\xi^{r-1}x}$ is all of $\G_m^r$. 
			From Lemma \ref{l: functional equation} we have
			\begin{align*}
				A(x) &= \AHp(x)^{n_0 + n_1\xi + \dots + n_{r-1}\xi^{r-1}}.
			\end{align*}
			As the $1, \dots, \xi^{r-1}$ are independent over $\Z$, we have $A(x)=\AHp(x)^n$, where $n$ is a nonzero $p$-adic number. Thus, it suffices to show $\AHp(x)^n$ is transcendental for any nonzero $p$-adic number with $p\nmid n$. Then from Lemma \ref{l: partials of powers of AH} we see that $\dlog{k}(\AHp(x)^n)$ has depth $k$. It follows from
			Corollary \ref{c: diff criterion using depths} that $\AHp(x)^n$ is transcendental over $\F_p(x)$. Thus, we have proven
			\begin{align*}
				\dim_{\F_p(x)} \F_p(x)(\AHp(\xi^i x))_{0\leq i \leq r-1} &= r.
			\end{align*}
			For $j>r-1$, we can write $\xi^j=c_0 + c_1 \xi + \dots + c_{r-1}\xi^{r-1}$. Then Lemma \ref{l: functional equation} gives 
			\begin{align*}
				\AHp(\xi^j x) &= \prod_{i=0}^{r-1} \AHp(\xi^i x)^{c_i},
			\end{align*}
			so that 
			\begin{align*}
			\dim_{\F_p(x)} \F_p(x)(\AHp(\xi^i x))_{0\leq i \leq p-2} &= r.
			\end{align*}
			\bibliographystyle{plain}
		\bibliography{bibliography.bib}
\end{document}